\documentclass[11pt,oneside,reqno]{amsart}
\usepackage{amsmath,amssymb,amsthm}
\usepackage{mathtools}
\usepackage{verbatim}
\usepackage{mathrsfs}
\makeatletter

\@namedef{subjclassname@2020}{\textup{2020} Mathematics Subject Classification}
\makeatother

\usepackage[colorlinks=true,linkcolor=blue,citecolor=magenta,pdfpagelabels=false]{hyperref}

\newcommand{\thmref}[1]{Theorem~\ref{#1}}
\newcommand{\lemref}[1]{Lemma~\ref{#1}}
\newcommand{\rmkref}[1]{Remark~\ref{#1}}

\newenvironment{acknowledgements}{\bigskip\textbf{Acknowledgements.}}{}

\numberwithin{equation}{section}

\theoremstyle{plain}
\newtheorem{theorem}{Theorem}[section]
\newtheorem{lemma}[theorem]{Lemma}
\newtheorem{corollary}[theorem]{Corollary}
\newtheorem{proposition}[theorem]{Proposition}

\theoremstyle{definition}
\newtheorem{definition}[theorem]{Definition}
\newtheorem{remark}[theorem]{Remark}

\newcommand{\Q}{{\mathbb Q}}

\newcommand{\Z}{{\mathbb Z}}

\newcommand{\F}{{\mathbb F}}

\newcommand{\lb}{\left(}
\newcommand{\rb}{\right)}

\newcommand{\ve}{\varepsilon}

\newcommand{\modd}{(\textnormal{mod }d)}
\newcommand{\be}{\begin{equation}}
	\newcommand{\ee}{\end{equation}}

\begin{document}
	\title[Coprimality of Fourier coefficients]{Coprimality of Fourier coefficients
		of 
		eigenforms}
	
	\author[S. Ganguly et. al.]{Satadal Ganguly, Arvind Kumar, Moni Kumari}
	\address{Theoretical Statistics and Mathematics Unit, Indian Statistical Institute, 203 Barrackpore
		Trunk Road, Kolkata-700108, India.}
	\email{sgisical@gmail.com}
	\address{Einstein Institute of Mathematics, the Hebrew University of Jerusalem, Edmund
		Safra Campus, Jerusalem 91904, Israel.}
	\email{arvind.kumar@mail.huji.ac.il}
	\address{Department of Mathematics, Bar-Ilan University, Ramat Gan 52900, Israel.}
	\email{moni.kumari@biu.ac.il}

	
	\subjclass[2020]{Primary: 11F30; Secondary: 11F80, 11N37, 11N64, 11N99}
	\keywords{Fourier coefficients of cusp forms, coprimality of  integer sequences, Chebotarev Density Theorem}
		
	\begin{abstract}
		Given a pair of distinct non-CM normalized eigenforms 
		having 
		integer Fourier coefficients $a_1 (n)$ and $a_2(n)$,
		we count positive integers $n$ with 
		$(a_1(n), a_2(n))=1$ and make a conjecture about the density of the set of primes $p$ for which $(a_1(p), a_2(p))=1$. We also study the average order of the number of 
		prime divisors of $(a_1(p), a_2(p))$. 
	\end{abstract}
	\maketitle
	\section{Introduction}
	\subsection{Motivation and the first result}
	Given two integer-valued sequences $a_1(n)$ and $a_2(n)$, an interesting question is, how the sequence of the $gcd$'s $(a_1(n), a_2(n))$ behaves; and in particular, how often $a_1(n)$ and $a_2(n)$ are coprime. 
	For example, if $a_1(n)=n$ and $a_2(n)=\phi(n)$, the Euler's $\phi$-function, then the density of such integers is zero. This follows from
	the  beautiful result of Erd\H{o}s \cite{erd} given below:
	\be \label{Erdos}
	\left|\{n \leq x: (n, \phi(n))=1\}\right|=\lb 1+o(1)\rb \frac{e^{-\gamma}x}{L_3 (x)},
	\ee
	where $\gamma$ is the Euler constant and $L_3(x)=\log \log \log x$. 
	In analogy with this result, V. Kumar Murty \cite{km} has shown that if $f$ is a normalized eigenform with integer Fourier coefficients $a_f(n)$, then
	\begin{equation}\label{vkmurty}
		\left|\{n \leq x: (n, a_f(n))=1\}\right|=O\lb \frac{x}{L_3(x)}\rb,
	\end{equation}
	where the implied constant depends on $f$.
	
	For integers $k \ge 2$ and $N\ge 1$, let $S_k(N)$ denote the space of holomorphic cusp forms (with trivial nebentypus)
	of weight $k$ and  level $N$.
	Suppose  $f_1 \in S_{k_1}(N_1)$ and 
	$f_2 \in S_{k_2}(N_2)$ are two distinct non-CM normalized eigenforms
	with rational integer Fourier coefficients $a_1(n)$ and $a_2 (n)$, respectively. Here, by ``eigenform" we mean eigenfunction of \underline{all} the Hecke operators (see
	\cite[\S 5.8]{DS}) and ``normalized" means $a_1(1)=a_2(1)=1$. We further assume that the two forms are not character twists of each other.
	The above will be our standing assumptions throughout the article.
	Now we ask, for what proportion of integers,
	$a_1(n)$ and $a_2 (n)$ are coprime?   The corollary to the theorem below gives a non-trivial answer.
	\begin{theorem}\label{main}
		Let $f_1$ and $f_2$ as above. Then we have,
		\begin{itemize}
			\item[(a)]
			\begin{equation}\label{gcd_n}
				|\{n\leq x: (n, (a_1(n), a_2(n)))=1\}| \ll \frac{x}{L_3(x)}.
			\end{equation}
			
			\item [(b)]
			For any positive integer $d >1$, we have,
			\begin{equation}\label{gcd_d}
				|\{n\leq x: (d, (a_1(n), a_2(n)))=1\}| \ll \frac{x L_3(x)}{L_2(x)},
			\end{equation}
			where $L_2(x) =\log \log x$.
		\end{itemize}
		Here the implied constant depends only on the two forms $f_1$ and $f_2$ in part {\rm (a)}, and only on $f_1, f_2$ and  $d$ in part {\rm (b)}.
	\end{theorem}
	From 
	part (b) of  Thm. \ref{main} the following is immediate (by choosing $d=2$, for example).
	\begin{corollary}\label{coprime_n}
		Let $f_1$ and $f_2$ be as in Thm. \ref{main}. Then  we have
		\be\label{eqno}
			|\{n\leq x: (a_1(n), a_2(n))=1\}| \ll \frac{x L_3(x)}{L_2(x)},
		\ee
		where  the implied constant depends only on the two forms $f_1$ and $f_2$. In particular, the density of the set of integers $n$ such that $a_1(n)$ and $a_2(n)$ are coprime is zero.
	\end{corollary}
	\begin{remark}
	 A definition prevalent in the literature (see, e.g.,  \cite[Chap.~6]{Iw}) is that a form $f$ is a Hecke eigenform or
a Hecke form if it satisfies the condition that it is a  common eigenfunction of the Hecke operators $T_n$ with $(n, N_f ) = 1$, where $N_f$ is
the level of the form $f$. This weaker condition does not, in general, imply the stronger condition that $f$ is a common eiegenfunction of all the Hecke operators though these two are equivalent in the case the form is primitive (i.e., a newform). In the above theorem, the condition that the two forms are common eigenfunctions of the Hecke
operators $T_n$ for all positive integers $n>1$ is essential since we need  multiplicativity of  the coefficients of the forms
in the proof (see \S 6). Since multiplicativity of the coefficients is not required in the proof of the other results stated below, the results remain true even if one 
takes the forms to be Hecke eigenforms in the sense described above. 
\end{remark}

\begin{remark}\label{asymp-problem}
A question that arises naturally is whether we can find  asymptotic formulae for the sums in \eqref{gcd_n}, \eqref{gcd_d}, and \eqref{eqno}, or for the simpler
sum in \eqref{vkmurty}. Obtaining asymptotics or even good lower bounds for these sums seems to be a very difficult task and any progress on this question, even under reasonable conjectures, will be very interesting. In this connection, one should recall that for the sum in \eqref{vkmurty}, the analogous question for  primes 
is a well-known unsolved problem. Indeed, for a 
non-CM eigenform $f\in S_k(N)$, it is expected that primes $p$ for which $(p, a_f(p))>1$ is extremely rare 
and a probabilistic model suggests that the number of
such primes up to $x$ should be of the order $O(\log \log x)$ (see \cite[\S 3]{gou}).  
\end{remark}
	\subsection{Restriction to primes: a probabilistic heuristic}
	Let $f_1$ and $f_2$ be as before.  Another interesting question is how frequently
	$a_1(p)$ and $a_2 (p)$ are coprime as $p$ varies over the primes. In other words, we are interested in the order of growth of the function
	\[
	C(x; f_1, f_2)=|\{p\leq x: (a_1(p), a_2(p))=1\}| .
	\]
	As there are eigenforms, e.g., the Ramanujan Delta function,
	whose all but finitely many Fourier coefficients at primes are even,  we 
	need to assume further that there are infinitely many primes $p$ such that one of $a_1(p)$
	and $a_2 (p)$ is odd. Under this condition, we  describe a probabilistic heuristic to guess the answer.  One of our crucial intermediate results is the asymptotic formula  \eqref{evl} which implies that for any fixed integer $m>1$, 
	\be\label{int}
	\pi_{f_1,f_2}(x,m) \sim \delta(m){\pi(x)}, \textnormal{  as  }x\rightarrow \infty.
	\ee 
	Here $	\pi_{f_1,f_2}(x,m)$ denotes the number of primes $p$ up to $x$ such that $m$ divides both $a_1(p)$ and $ a_2(p)$; and $\delta=\delta_{f_1, f_2}$ is an arithmetical function 
	determined by the two forms $f_1$ and $f_2$ which we have defined and studied in some depth in \S \ref{S_algebraic}. 	In view of the  asymptotic relation \eqref{int},  $\delta(m)$ can be interpreted as the ``probability" (in the sense of density) that a ``random" prime $p$ has the property that $m$ divides both $a_1(p)$ and $a_2(p)$. 	
 Therefore, assuming that the conditions of divisibility by different primes do not influence each other
	(i.e., the ``events" are independent), it seem 
	reasonable to conjecture that the function $\delta$ should be multiplicative. Indeed, we have shown that $\delta$ is multiplicative on the set of integers that are supported on primes that are sufficiently large (see Prop. \ref{delta_multiplicative_large_prime}) for forms of general level and we have been able to establish multiplicativity over the entire set of natural numbers when the forms have level one (see Prop. \ref{delta_multiplicicative}).
	Moreover, the above assumption of independence suggests that the ``probability" that for a ``random" prime $p$, $a_1(p)$ and $a_2(p)$ are coprime should be given by the infinite product $\alpha=\alpha_{f_1, f_2}$ defined by
	\be\label{alpha}
	\alpha :=\prod_{\ell \  \rm{prime}} (1-\delta(\ell)).
	\ee
 Since we know that the sum $\sum_{\ell}\delta(\ell)$ converges (see Prop. \ref{asymptotic_delta})
	and that $0<\delta(\ell)<1$ for all primes $\ell$ (see \S \ref{S_algebraic}), 	it is clear that the above infinite product converges to some real number in the interval $(0, 1)$. The above discussion leads us to  the following.\\
	\noindent
	\textbf{Conjecture}: Under the assumptions stated above, we have the asymptotic relation
	\be \label{conj}
	C(x; f_1, f_2)= |\{p\leq x: (a_1(p), a_2(p))=1\}| \sim \alpha \pi (x) \ \ \textnormal{as} \  \     x\rightarrow \infty,
	\ee
	where $\pi(x)$ denotes, as usual, the number of primes up to $x$. 

		Proving this conjecture, even under GRH (the Generalized Riemann Hypothesis) for  all Artin $L$-functions, seems to be out of the reach of the current knowledge. The difficulty lies in the fact that the error term in  \eqref{evl} grows too rapidly in terms of $m$ and after expressing the coprimality condition using the M\"{o}bius function, the error term becomes unmanageable. See also Remark \ref{difficulty}. However, it 
		is still possible to give an upper bound of the expected order of magnitude under  GRH as stated below. 

\begin{theorem}\label{sieve}
Under the above assumptions on two forms and under GRH, one has the upper bound
\be\label{ub}
C(x; f_1, f_2)\leq (\alpha'+o(1))\pi (x),
\ee
where 
\[
 \alpha'= \sum_{n=1}^\infty \mu(n)\delta(n).
\]
\end{theorem}	
	Note that $\alpha'=\alpha$ if $\delta$ is multiplicative on the full set of positive integers, and hence, in particular, when the two forms are of level one.
	\begin{remark}\label{quasi}
	We remark that the full strength of GRH is not essential to prove the above theorem. Indeed, an analysis of the proof shows that a quasi-GRH, 
	which refers to the assertion that no Artin $L$-function has a 
	zero in the region $\Re(s) >1-\delta $  for some fixed $\delta$ with $0<\delta\leq 1/2$, is sufficient for our purpose. However, one should note that
 the exponent of 
	$x$ in the error term in \eqref{effective_cdt} will now depend on $\delta$ (see Remark \ref{quasi-Chebo}), and therefore the error terms in part (b) of both Prop. \ref{asymptotic_pi} and Prop. \ref{asymptotic_pi*} will change accordingly. 
	\end{remark}
	
	\subsection{Numerical verification of the conjecture}
	We have tested the conjecture numerically using SAGE \cite{sage}
	for every pair of  eigenforms $f_i$ and $f_j$ $(1\le i,j\le 3, i \neq j)$, where
	\begin{align*}
		&f_1(z)= q - 4q^{2} - 15q^{3} - 16q^{4} - 19q^{5} + 60q^{6} 
		+ \cdots\in S_6(11),\\
		&f_2(z)= q - 5q^{2} - 7q^{3} + 17q^{4} - 7q^{5} + 35q^{6}
		+ \cdots \in S_4(13),\\
		&f_3(z)= q + 10q^2 - 73q^3 - 28q^4 - 295q^5 - 730q^6 
		+ \cdots 
		\in S_8(13).
	\end{align*}
	Let us denote  the $n$th Fourier coefficient of $f_i$ by  $a_i(n)$, for $1\le i\le 3$. For a large integer $x$, the ratio $R(x; f_i, f_j):=|\{p\leq x: (a_i(p), a_j(p))=1\}|/\pi(x)$ is an approximation of the density of the set of primes for which $a_i(p)$ and $a_j(p)$ are coprime and  the data, for different pairs $(f_i, f_j)$, are presented in the second column of the table below, where we have taken $x=$100000.
	Note that for a large prime $\ell$, $\delta(\ell)$ is close to zero and so only the first few terms in the infinite product $\alpha=\prod_{\ell}(1-\delta(\ell))$ should make a significant 
	contribution as the other factors are very close to $1$. Thus, we can approximate  $\alpha$ by taking the product of primes $\ell$ up to $L$, for a suitably large integer $L$.
	Due to constraints on our computational resources we took $L=100$.  Furthermore, since $\delta$ is hard 
	to compute on the exceptional primes (see \S \ref{S_algebraic}) and finding the exceptional primes is a difficult task in itself, we have
	also approximated  $\delta(\ell)$, for primes $\ell$ up to $L$, by the ratio $\delta (y, \ell):= \pi_{f_i,f_j}(y,\ell)/ {\pi(y)}$ for a large integer $y$. Thus $\alpha_{L, y}(f_i, f_j):=\prod_{\ell \leq L}(1-\delta_y(\ell))$ should be a crude approximation of the constant $\alpha_{f_i, f_j}$ and the third column gives this approximated value for different pairs $(f_i, f_j)$, where we have taken $y=$50000. 
	The closeness of the data in columns 2 and 3 of the table inspires some confidence in
	the truth of this conjecture but further numerical investigations will be welcome. 
	\begin{center}
		Table I: Numerical examples in support of Conjecture \ref{conj} \\
		$(x; y; L)=(10^5; 5. 10^4; 100)$\\
		
		\begin{tabular}{|p{4.0cm}|p{4.0cm}|p{4.0cm}|}
			\hline
			Pair of  forms $(f_i,f_j)$	&  {$R(x; f_i, f_j)$}& $ \alpha_{L, y}(f_i, f_j)$ \\
			\hline\hline
			$(f_1,f_2)$
			& 0.40763
			& 0.40757 \\ 
			\hline 	
			{$(f_1,f_3)$} & 0.42212 
			& 0.42414  \\ 
			\hline
			{$(f_2,f_3)$} 
			&  0.13178 & 0.13265 \\
			\hline
		\end{tabular}
	\end{center}

	\subsection{Number of prime divisors of $(a_1(p), a_2(p))$}
	We have studied some related questions
	about the sequence $\{(a_1(p), a_2(p)): p \textnormal{ prime}\}$ which are interesting in their own right. The first one concerns
	the average order of $\omega ((a_1(p), a_2(p)))$, where $\omega(n)$ denotes
	the number of distinct prime divisors of $n$.
	
	If $\alpha(n)$ is an arithmetic function, then it is natural to study the growth of the function $\omega(\alpha(n)).$  For instance, it was shown by Murty and Murty \cite[Thm. 6.2]{mm} that
	both $\omega(\phi(n))$ and $\omega(\sigma(n))$,  where $\sigma(n)$ is the sum of divisors of $n$, have normal order $\frac 1 2 (\log \log n)^2$, and  if $a_f(n)$'s are the Fourier coefficients of a normalized eigenform $f$, then under GRH, Murty and Murty \cite{mm} proved that $\omega(a_f(p))$ has normal order $\log \log p$, where $p$ runs over the set of primes.  In a subsequent paper \cite{mm2}, they proved an analogue of the Erd\H{o}s-Kac Theorem for $\omega(a_f(p))$ as $p$ runs
	over primes. Inspired by these results we have studied  the sequence $\omega((a_1(p),a_2(p)))$, as $p$ varies over the set of primes. 
	Before stating our main result we first make a few comments to put matters into perspective. We shall use the notation $\sum^{'}$ to denote sums over primes $p$ with $a_1(p)a_2(p) \neq0$.
	First of all, 	we recall a classical bound which follows easily from PNT but can be proved independently (see \cite[\S 5]{Ramanujan}):
	\begin{equation} \label{on}
		\omega{(n)}\ll \frac{\log n}{\log \log n},
	\end{equation} 
	where the implied constant is absolute;
	and using this we easily obtain the unconditional bound
	\begin{equation*}
		\sideset{}{'}\sum\limits_{p\le x} \omega((a_1(p),a_2(p)))\ll \frac{x}{\log \log x},
	\end{equation*}
	where the implied constant depends only on the two forms. 
	By the work of Murty and Murty (see \cite[Thm. 3.1]{mm}), we  have, under GRH,
	\begin{equation*}
		\sideset{}{'}\sum\limits_{p\le x} \omega((a_1(p),a_2(p)))\ll \sideset{}{'}\sum\limits_{p\le x} \omega(a_1(p))\ll \frac{x \log \log x}{\log x}.
	\end{equation*}
	We now state our main result.
	
	\begin{theorem}\label{omega} With the same assumptions and  notation as above and under GRH, we have the bounds

		\begin{equation*}
			\sideset{}{'}\sum\limits_{p\le x} \omega((a_1(p),a_2(p))), \ \ \sideset{}{'}\sum\limits_{p\le x} \omega^2((a_1(p),a_2(p)))\ll \frac{x}{\log x},
		\end{equation*}
		where the implied constant depends only on the two forms. 
	\end{theorem}
	
	\begin{corollary}
		For any function $h:\mathbb{N}\rightarrow [0, \infty)$ that increases to infinity, however slowly, the subset of primes
		$\{p: (a_1(p), a_2(p))> h(p)\}$ has density zero. 
	\end{corollary}
	
	\begin{remark}
		The upper bound in the above theorem is of the right order of magnitude as  the lower bound is also $\frac{x}{\log x}$. This is clear from  \eqref{zero_coefficient}.
		Thus, we have, in fact, 
		\[
		\sideset{}{'}\sum\limits_{p\le x} \omega((a_1(p),a_2(p))) \asymp \frac{x}{\log x}, \  \sideset{}{'}\sum\limits_{p\le x} \omega^2((a_1(p),a_2(p))) \asymp \frac{x}{\log x}.
		\]
	\end{remark}
	Obtaining  a precise asymptotic formula for these sums appears to be quite difficult, but
	if we replace the function $\omega$ with a function that counts only  small enough prime divisors of $(a_1(p),a_2(p))$, then we can extract a main term.
	To be precise,  let us  define for a positive real number $u$,
	\[
	\omega_u (n):=\sum_{p | n,\  p\leq u} 1.
	\]
	Then we have:
	\begin{theorem}\label{omega_u}
		Under GRH, and under the same assumptions and with the same notation as above,   there exist  explicit constants $c_1, c_2 >0$ such that for any $\ve>0$ we have the following. 
		\begin{itemize}
			\item[(a)]
			For any $u, \  x^\ve \le u<x^{1/12-\ve}$,
			\begin{equation*}
				\sideset{}{'}\sum\limits_{p\le x} \omega_u((a_1(p),a_2(p)))=c_1 \frac{x}{\log x}+O(x^{1-\ve}).
			\end{equation*}
			\item[(b)]
			For any $u, \ x^\ve \le u<x^{1/24-\ve}$,
			\begin{equation*}
				\sideset{}{'}\sum\limits_{p\le x} \omega_u^2((a_1(p),a_2(p)))=c_2 \frac{x}{\log x}+O(x^{1-\ve}).
			\end{equation*}
		\end{itemize}
		Furthermore, the constants $c_1$ and $c_2$ are given by
		\[
		c_1=\sum_{\ell}\delta(\ell),
		\]
		and
		$$
		c_2=\sum_{\substack{\ell_1,\ell_2\\ \ell_1\neq \ell_2}} \delta(\ell_1\ell_2)+ \sum_{\ell}\delta(\ell),
		$$
		where $\delta$ is defined by  \eqref{delta_definition} and $\ell, \ell_1, \ell_2$ run over the set of all primes. The implied constants in both {\rm (a)} and {\rm (b)} 
		depend only on the two forms. 
	\end{theorem}
	\begin{remark}\label{quasi2}
	As in the case of Theorem \ref{sieve}, an analysis of the proof shows that a quasi-GRH statement is sufficient for  proving Theorem \ref{omega} and Theorem \ref{omega_u}. 
	\end{remark}

	\subsection{Structure of the paper and the basic ideas of the proofs}
	The main tools used as black boxes in the proofs are 
	Deligne's theorem on Galois representations attached to eigenforms and the Chebotarev Density Theorem, especially the effective version due to Lagarias and Odlyzko. These results are recalled in \S \ref{S_basic}. In \S \ref{S_algebraic}, we use the above machinery to determine the size of the image of the mod-$m$ Galois representation attached to the pair of eigenforms and to obtain the asymptotic size of the function $\delta(m)$. Here the results of Ribet and Loeffler on the image of the  Galois representations associated to a collection of modular forms play a major role. 
	The main result in the next section is an asymptotic formula (see  Prop. \ref{asymptotic_pi}) for the number of primes $p$ up to $x$ such that both $a_1 (p)$ and $a_2 (p)$ are divisible by a fixed positive integer $m$. The  idea here is to translate the divisibility condition into a statement about the traces of $\bar\rho_{f_i,m}(\textnormal{Frob}_p)$, where $\bar\rho_{f_i,m}$ denotes mod-$m$ Galois representation associated to the form $f_i$ by Deligne's theorem. Thereafter, in \S \ref{S_preparation}, we again obtain some intermediate technical results of analytic nature which are required in the latter sections. Finally, in \S \ref{S_main}, \S \ref{proof-sieve} and \S \ref{S_omega}, we finish the proofs by  applying  techniques from Analytic Number Theory;
in particular,  those  developed and employed by by P. Erd\H{o}s,  V. Kumar Murty and M. Ram Murty (see, e.g.,  \cite{erdo}, \cite{erd}, \cite{mm}, \cite{mm2}, \cite{km}).

	\subsection{Notation and conventions}
	By density
	of a subset $S$ of the set of primes, we mean the natural density; i.e., $\lim_{x\rightarrow \infty} \frac{|S\cap[1, x]|}{\pi(x)}$
	($|A|$ denotes the size of a subset $A$ of $\mathbb{N}$), 
	if the limit exists.
	Here $\pi(x)$ denotes the number of primes less than or equal to $x$ for 
	any real number $x\ge 2$.
	The letter $\ve$ will denote a positive real 
	number which can be taken to be as small as we want and in different occurrences it may assume different values. 
	The notation ``$f(y) =O(g(y))$" or ``$f(y) \ll g(y)$", where $g$ is a positive function, mean
	that there is a constant $c>0$ such that $|f(y)|\leq cg(y)$ for any $y$ in the concerned domain. 
	The dependence of this implied constant on some parameter(s) may sometimes
	be displayed by  a suffix (or suffixes) and may sometimes be suppressed but it will be clear from the context. For example, the implied constant will often depend on the pair of forms under consideration. For two positive functions $f$ and $g$, the notation $f(x) \asymp g(x)$ means both the bounds
	$f(x) \ll g(x)$ and $g(x) \ll f(x)$ hold. 
	The notation ``$f(y)=o(g(y))$", where $g$ is a positive function, means that $f(y)/g(y) \rightarrow 0$
	as $y\rightarrow \infty$ and ``$f(y) \sim g(y)$" means $f(y)-g(y) =o(g(y))$. The letters
	$p, q, \ell,\ell_1, \ell_2$ etc.~will denote prime numbers throughout. We write $L_1(x)=\log x$ and for $i\ge 2$, define $L_i(x)$ inductively  by $L_i(x)=\log L_{i-1}(x).$ PNT and CDT denote the Prime Number Theorem and the Chebotarev Density Theorem, respectively.

	\section{Background materials}\label{S_basic}
	\subsection{Chebotarev Density Theorem}
	Let $K$ be a finite Galois extension of $\Q$ with the Galois group $G$ and degree $n_K$. 
	For an unramified prime $p$, we denote by ${\rm Frob}_p$, a Frobenius element of $K$ at $p$ in $G$.
	For a subset $C$ of $G$, stable under conjugation, we define 
	$$ 
	\pi_C(x):=\{p\le x: p~ {\rm unramified ~in}~ K ~{\rm and ~Frob}_p\in C\}.
	$$
	Let $d_K$ denote the absolute value of the discriminant of $K/\Q$. An effective version of  the Chebotarev Density Theorem (denoted by CDT henceforth) was established  by Lagarias and Odlyzko \cite{lo} and it asserts that there is a constant $c_1>0$ such that  for every $x\geq 2$ with $\log x \geq c_1 n_K (\log d_K)^2 $, we have
	\begin{equation}\label{cdt}
		\left|\pi_C(x)- \frac{|C|}{|G|}{\rm Li}(x) \right| \leq  \frac{|C|}{|G|}{\rm Li}(x^\beta)+O\left( \Vert C\Vert x ~{\rm exp}\left(-c'\sqrt{\frac{\log x}{n_K}}\right) \right),
	\end{equation}
	where  the constant $c'>0$ is effectively computable, $\Vert C\Vert$ denotes the number of conjugacy classes contained in $C$ and ${\rm Li}(x)=\int_2^x \frac{dt}{\log t}$, the logarithmic integral function. Here  $\beta$ is the possibly existing real ``exceptional zero'', also called ``the Landau-Siegel zero" of  the Dedekind zeta function $\zeta_K(s)$ in the strip 
	$$
	1-\frac{1}{4\log d_K}\le \Re (s)<1;
	$$
	and if $\beta$ does not exist, then the corresponding term is omitted from  \eqref{cdt}. By the works of Heilbronn \cite{hei} and Stark \cite{sta}, we know
	that (see \cite[Eq. (27)]{sta})
	\begin{equation}\label{beta_stark}
		\beta \le 1-\frac{c_0}{d_K^{1/n_K}},
	\end{equation}
	where $c_0>0$ is an effective constant. We also need a conditional version of  CDT, that is, under the assumption of GRH. 
	This was first obtained by Lagarias and Odlyzko (op. cit.). We quote the following from \cite[Thm. 4]{ser}. 
	\begin{proposition}
		Suppose the Dedekind zeta function $\zeta_K(s)$ 
		satisfies the Riemann Hypothesis. Then for every $x\geq 2$, 
		\begin{equation}\label{effective_cdt}
			\pi_C(x)= \frac{|C|}{|G|}\pi(x)+O\bigg(\frac{|C|}{|G|}x^{{1}/{2}}(\log d_K+n_K\log x)\bigg).
		\end{equation}
	\end{proposition}
	\begin{remark}\label{quasi-Chebo}
	 If we assume a milder version of the hypothesis; namely that $\zeta_K(s)$ does not vanish if $\Re (s)>1-\delta$ for some $\delta \in (0, 1/2)$
	then we obtain a similar asymptotic formula, the only difference being that the exponent of $x$ becomes $1-\delta$.
	\end{remark}
	
	\subsection{mod-$m$ Galois representations}
	In this section, we recall some of the fundamental results on 
	Galois representations associated with modular forms. 
	Let $G_{\Q}= {\rm Gal}(\bar\Q/\Q)$ be the Galois group of an algebraic closure $\bar\Q$ of $\Q$. The following
	result is due to Deligne.
	\begin{theorem}\cite{del}
		Let $k\ge 2, N\ge 1$ and let $ \ell $ be a prime. Then to any normalized eigenform $f\in S_k(N)$ 
		with integer Fourier coefficients $a_f(n)$  one can attach a continuous two-dimensional 
		Galois representation of the rationals
		\begin{equation*}
			\rho_{f, \ell }:G_{\Q}\rightarrow {\rm GL}_2(\Z_\ell )
		\end{equation*}
		such that $\rho_ {f,\ell}$ is odd and irreducible. Also, for all primes $p\nmid N\ell$
		the representation $\rho_ {f,\ell}$ is unramified at $p$ and  
		$$ {\rm{tr}}(\rho_ {f,\ell}({\rm{Frob}}_p))=a_f(p), ~~~~ {\rm {det}}(\rho_ {f,\ell}({\rm{Frob}}_p))=p^{k-1}.$$
	\end{theorem}
	
	\noindent
	By reduction and semi-simplification, we obtain a mod-$\ell$ Galois representation, namely 
	$$ {\overline{\rho}}_ {f,\ell}:G_{\Q}\rightarrow {\rm GL}_2(\F_\ell), $$
	where $\F_{\ell}:=\Z/\ell\Z$.
	
	Let $m$ be a positive integer with prime factorization $m= \prod_{j=1}^{r}\ell_j^{e_j}$.
	Using the $\ell_j$-adic representations associated to $f$, 
	we construct an $m$-adic representation 
	\begin{equation*}
		{\rho}_{f,m}:  G_\Q \rightarrow {\rm GL}_2\big(\prod_{1\le j\le r}\Z_{\ell_j}\big).
	\end{equation*}
	
	For each $1\le j\le r$, we have the natural projection $\Z_{\ell_j}\twoheadrightarrow \Z/{\ell_j^{e_j}}\Z$,
	and hence we get the reduction ${\overline{\rho}}_{f,m}$ of $m$-adic representation given by 
	
	$$ {\overline{\rho}}_{f,m} :G_\Q \rightarrow
	{\rm GL}_2(\prod_{1\le j\le r}  \Z/{\ell_j^{e_j}}\Z) \xrightarrow{\cong} {\rm GL}_2(\Z/m\Z).$$
	Furthermore, if $p\nmid mN$ is a prime, then $\bar{\rho}_{f,m}$ 
	is unramified at $p$ and
	$$ {\rm tr}\left(\bar{\rho}_{f,m}\lb{\rm Frob}_p\rb\right)\equiv a_f(p) \pmod m,~~~~ 
	{\rm det}\left(\bar{\rho}_{f,m}\lb{\rm Frob}_p\rb\right)\equiv p^{k-1}\pmod m. $$
	
	\section{Algebraic preliminaries}\label{S_algebraic}
	Let $f_1 \in S_{k_1}(N_1)$ and $f_2 \in S_{k_2}(N_2)$ be as in the introduction and suppose they have integer Fourier coefficients $a_1(n)$ and $a_2(n)$, respectively.
	Suppose  $\bar{\rho}_{f_i,m}: G_\Q \rightarrow {\rm GL}_2(\Z/m\Z)$, for $i=1,2$, denotes the mod-$m$ Galois representation associated to $f_i$. Then we consider the product 
	representation $\bar{\rho}_m$ of $\bar{\rho}_{f_1,m}$ and $\bar{\rho}_{f_2,m}$, defined by
	\begin{align*}
		\bar{\rho}_m: G_\Q & \rightarrow {\rm GL}_2(\Z/m\Z)\times {\rm GL}_2(\Z/m\Z),\\
		\sigma & \mapsto (\bar{\rho}_{f_1,m}(\sigma), \bar{\rho}_{f_2,m}(\sigma)) \notag.
	\end{align*}
	Let $\mathscr{A}_m$ denote the image of $G_\Q$ under $\bar{\rho}_m$ and
	let $$ H_m:= {\rm Ker}(\bar{\rho}_m)= \{ \sigma \in G_\Q : \bar{\rho}_{f_1,m}(\sigma)=
	\bar{\rho}_{f_2,m}(\sigma)= {\rm Id} \},$$
	where Id denotes the identity element of the group ${\rm GL}_2(\Z/m\Z)$. 
	Therefore,
	\begin{equation}\label{id1}
		\frac{G_\Q}{H_m} \cong \mathscr{A}_m. 
	\end{equation}
	Since $\bar{\rho}_{f_i,m}$ is a continuous homomorphism for each $i=1,2$, $\bar{\rho}_m$ 
	is continuous. Hence
	$H_m$ is an open and closed normal subgroup of $G_\Q$  (the target group of $\bar{\rho}_{f_i,m}$ being equipped with the discrete topology). 
	By the fundamental theorem of Galois theory,  the fixed field of $H_m$, say $L_m$, is a finite 
	Galois extension of $\Q$ and
	\begin{equation}\label{id2}
		\frac{G_{\Q}}{H_m}\cong {\rm Gal}(L_m/ \Q).
	\end{equation}
	Combining \eqref{id1} and  \eqref{id2}, we have
	\begin{equation}\label{image_isomorphism}
		{\rm Gal}(L_m/ \Q) \cong \mathscr{A}_m.
	\end{equation}
	Let $\mathscr{C}_m$ be a subset of $\mathscr{A}_m$ defined by
	\begin{equation}\label{definition_C_m}
		\mathscr{C}_m=\{(A,B)\in \mathscr{A}_m: {\rm tr}(A)={\rm tr}(B)=0\}.
	\end{equation}
	Let us now define the following function on the set of positive integers which plays an important role in this work.
	\begin{definition}
		For an integer $m >1$, define
		\begin{equation}\label{delta_definition}
			\delta(m):=\frac{|\mathscr{C}_m|}{|\mathscr{A}_m|}
		\end{equation}
		and $\delta(1):=1$. 
	\end{definition}
	
	\begin{remark}\label{delta_positive}
		Since the trace of the 
		image of complex conjugation is always zero, 
		$\mathscr{C}_m\neq \phi$, and hence $\delta(m)>0$ for every integer $m$. Furthermore, because the identity element lies in $\mathscr{A}_\ell$ but not in $\mathscr{C}_\ell$ for any odd prime $\ell$, we always have
		$ \delta(\ell)<1$. The prime $2$ is special and there are normalized eigenforms all whose coefficients at primes are even; e.g., the Ramanujan Delta function. 
	\end{remark}

	\subsection{Sizes of  $\mathscr{A}_{\ell}$ and  $\mathscr{C}_{\ell}$ for large primes $\ell$}\label{delta-asymp}
	It is clear that for any prime $\ell$ and any integer $n\geq 1$, $\mathscr{A}_{{\ell}^n}$ is contained in the set
	\begin{align*}
		\{(A, B)\in {\rm GL}_2(\Z/\ell^n\Z)\times {\rm GL}_2(\Z/\ell^n\Z): 
		{\rm det}(A)=v^{k_1 -1}, {\rm det} (B)= v^{k_2 -1}, v \in {(\Z/\ell^n\Z)}^{\times} \}
	\end{align*}
	and it follows from the work of Loeffler \cite[Thm. 3.2.2]{loe} (see also \cite[Thm. (6.1)]{rib} 
	for an earlier result 
	for forms of level $1$), that 
	outside a finite set of exceptional  primes  determined by the two forms, $\mathscr{A}_{{\ell}^n}$ is exactly the above set; i.e., the image of the product mod-${{\ell}^n}$ representation of two eigenforms (that are not twists of each other)
	is as large as possible if $\ell$ does not belong to a finite set. In other words, there is a positive integer $M=M(f_1, f_2)$ such that for every prime $\ell>M$ and every integer $n\geq 1$, we have,
	\begin{align}\label{image}
		\mathscr{A}_{{\ell}^n}=\{(A, B)\in {\rm GL}_2(\Z/\ell^n\Z)\times {\rm GL}_2(\Z/\ell^n\Z): 
		{\rm det}(A)=v^{k_1 -1},   {\rm det} (B)= v^{k_2 -1}, v \in {(\Z/\ell^n\Z)}^{\times} \}.
	\end{align}
	Clearly, we also have, for $\ell>M$,
	\begin{align}\label{image_C}
		\mathscr{C}_{\ell^n} = \{(A, B)\in \mathscr{A}_{{\ell}^n}: 
		{\rm tr}(A)={\rm tr}(B)=0 \}.
	\end{align}
	
	In the next two lemmas, we compute the cardinalities of $\mathscr{A}_{\ell}$ and $\mathscr{C}_{\ell}$ for  $\ell>M$. We first set
	$$
	d=(\ell-1,k_1-1,k_2-1).
	$$
	
	\begin{lemma}\label{A_l_computation}
		For  any prime $\ell > M$,
		$$ | \mathscr{A}_{\ell} | =\frac{1}{d}(\ell-1)^3(\ell^2+\ell)^2.$$
		
	\end{lemma}
	\begin{proof}
		First we consider the group $\Lambda$, defined by
		\begin{equation}\label{Lambda}
			\Lambda=\{ (v^{k_1-1},v^{k_2-1}): v\in \mathbb F_\ell^{\times}  \}.
		\end{equation}
		Therefore
		\begin{align*}
			|\mathscr{A}_{\ell}|&= \sum_{(t_1,t_2)\in \Lambda} |\{(A, B)\in {\rm GL}_2(\F_\ell)\times {\rm GL}_2(\F_\ell): 
			({\rm det}(A), {\rm det} (B))= (t_1,t_2)\}|\\
			&=  \sum_{(t_1,t_2)\in \Lambda} |\{A\in {\rm GL}_2(\mathbb F_\ell): {\rm {det}}(A)=t_1\}|~ |\{B\in {\rm GL}_2(\mathbb F_\ell): {\rm {det}}(B)=t_2\}|.
		\end{align*}
		Since 
		$$
		{\rm GL}_2(\mathbb F_\ell) = \bigcup\limits_{t\in \mathbb F_\ell^\times} \begin{pmatrix}
			t&0\\0&1
		\end{pmatrix} {\rm SL}_2(\mathbb F_\ell),
	$$
	   it follows that for any fixed $t\in \mathbb F_\ell^{\times}$, the cardinality of the set
		$\{A\in {\rm GL}_2(\mathbb F_\ell): {\rm {det}}(A)=t\}$ does not depend on $t$ and is equal to
		$|{\rm SL}_2(\mathbb F_\ell)|=(\ell-1)(\ell^2+\ell)$. Thus
		\begin{align*}\label{ac1}
			| \mathscr{A}_{\ell} |  =(\ell-1)^2(\ell^2+\ell)^2 |\Lambda|.
		\end{align*}
		To compute the cardinality of $\Lambda$, consider the surjective group homomorphism 
		\begin{equation*}
			\phi:  \mathbb F_\ell^{\times}\rightarrow \Lambda 
			{\rm ~defined ~by} ~\phi(v)=  (v^{k_1-1},v^{k_2-1}).
		\end{equation*}
		By using the fact that $d=(\ell-1,k_1-1,k_2-1)$ one can easily see that ${\rm{Ker}}(\phi)=\{v \in \mathbb F_\ell^{\times}: v^{d}=1\}$, a cyclic subgroup of $\mathbb F_\ell^\times$ of order $d$. Therefore
		\begin{equation}\label{Lambda_cardinality}
			|\Lambda|=\frac{1}{d}(\ell-1)
		\end{equation}
		and this completes the proof.
	\end{proof}
	
	\begin{lemma}\label{C_l_computation}
		For any prime $\ell > M$,
		$$ | \mathscr{C}_{\ell} | =\frac{1}{d}\ell^2(\ell-1)(\ell^2+1).$$
		
	\end{lemma}
	
	\begin{proof}
		From the definition of $\mathscr{C}_{\ell}$, we can write
		\begin{align}\label{clc1}
			|\mathscr{C}_{\ell}|
			=  \sum_{(t_1,t_2)\in \Lambda} |\{A\in {\rm GL}_2(\mathbb F_\ell): {\rm {det}}(A)=t_1, {\rm tr}(A)=0\}| ~  |\{B\in {\rm GL}_2(\mathbb F_\ell): {\rm {det}}(B)=t_2, {\rm tr}(B)=0\}|,
		\end{align}
		where $\Lambda$ is defined by  \eqref{Lambda}. For $t \in\mathbb F_\ell^{\times}$, we can easily obtain the following equality by an elementary
		counting argument: 
		\begin{align}\label{id11}
			&|\{A\in {\rm GL}_2(\mathbb F_\ell): ~{\rm {det}}(A)=t, {\rm tr}(A)=0\}| =
			\begin{cases}
				\ell^2+\ell, & -t {\rm~is~ quadratic~ residue},\\
				\ell^2-\ell, & {\rm otherwise}. 
			\end{cases}
		\end{align}
		Let ${(t_1,t_2)\in \Lambda}$. Since $k_1$ and $k_2$ are even, we see that 
		$-t_1$  is a quadratic residue (respectively, non-residue) if and only if $-t_2$ is a quadratic residue (respectively, non-residue).   We split the sum on the right hand side of   \eqref{clc1} into two parts depending on
		$-t_1$ is a quadratic residue or not and obtain 
		\begin{align}\label{Cl2}
			| \mathscr{C}_{\ell} |=&(\ell^2+\ell)^2 \sum_{\substack{(t_1,t_2)\in \Lambda\\ -t_1:~ {\rm residue}}} 1
			+(\ell^2-\ell)^2 \sum_{\substack{(t_1,t_2)\in \Lambda\\ -t_1:~ {\rm non-residue}}}1.
		\end{align}
		Since the group homomorphism 
		$$
		\Lambda \rightarrow \{\pm 1 \} {\rm ~defined ~by} ~(t_1,t_2)\mapsto \left(\frac{t_1}{\ell}\right)
		$$
		is surjective, where $\left(\frac{\cdot}{\ell}\right)$ denotes the Legendre symbol, the subgroup consisting of quadratic residue elements of $\Lambda$ is of index two. Therefore,
		$$
		\sum_{\substack{(t_1,t_2)\in \Lambda\\ -t_1:~ {\rm residue}}} 1
		=\sum_{\substack{(t_1,t_2)\in \Lambda\\ -t_1:~ {\rm non-residue}}}1=\frac{|\Lambda|}{2}= \frac{1}{2d}(\ell -1),
		$$
		where we have used  \eqref{Lambda_cardinality} in the last equality. Substituting this in  \eqref{Cl2} gives the desired result.
	\end{proof}

	We record two simple consequences of the two  foregoing lemmas in the following proposition.
	\begin{proposition}\label{asymptotic_delta}
		For every prime $\ell > M$,
		\be\label{deltabd}
		\delta(\ell) \leq \frac{3}{\ell^2},
		\ee
		and as
		$\ell$ varies over
		primes,   
		\begin{equation*}
			\delta(\ell)\sim \frac{1}{\ell^2}\mathrm{ \ \  as \ \  } \ell \rightarrow \infty.
		\end{equation*}
	\end{proposition}
	
	\subsection{Multiplicativity of $\delta(m)$}
	From the explicit descriptions of $\mathscr A_{\ell}$ and $\mathscr C_{\ell}$ for large primes $\ell$ given in  \eqref{image} and  \eqref{image_C},  it is clear that (see, e.g., \cite[Lemma 5.4]{mm})
	if $f_1$ and $f_2$ are two eigenforms as before then the following result holds:
	\begin{proposition}\label{delta_multiplicative_large_prime}
		For all  primes $\ell_1, \ell_2 >M$ with $\ell_1 \neq \ell_2$ and any positive integers $n_1$ and $n_2$, we have
		$$
		\delta(\ell_1^{n_1} \ell_2^{n_2})=\delta(\ell_1^{n_1}) \delta(\ell_2^{n_2}).
		$$
	\end{proposition}
	\noindent
	However, in the following we prove that if the forms are of level 1, then the function $\delta$ is multiplicative on the entire set of positive integers.  This result is of independent interest and it may be useful in other investigations.
	\begin{proposition}\label{delta_multiplicicative}
		Let $f_1$ and $f_2$ be two normalized eigenforms of level $1$ and both have rational integral
		coefficients. Then $m \mapsto \delta(m)$ is multiplicative; i.e., if $m=m_1 m_2$ and 
		$(m_1, m_2)=1$, then 
		$$\delta(m)=\delta(m_1)\delta(m_2).$$
	\end{proposition}
	\begin{proof}
		From the definition of the function $\delta$, it is sufficient to show that as a function of $m$,
		$|\mathscr A_m|$ and $|\mathscr C_m|$ are multiplicative. 
		We first show that the function
		$m \mapsto |\mathscr{A}_m|$ is multiplicative. By \eqref{image_isomorphism} and the 
		fundamental theorem of 
		Galois theory, we know that
		\begin{equation}\label{ah}
			|\mathscr{A}_m|= [L_m:\Q].
		\end{equation}
		For $i=1,2$, $\bar{\rho}_{m_i}$ is ramified only at the primes dividing $m_i$ (since the level is $1$), and therefore,
		it follows that $L_{m_i}$ is ramified only at the primes dividing $m_i$. Therefore,
		\begin{equation}\label{lh1}
			L_{m_1}\cap L_{m_2} =\Q.
		\end{equation}
		Now, from the definition of mod-$m$ representations, we have $H_{m_1 m_2}=H_{m_1} \cap H_{m_2}$ and
		it easily follows that  $L_m$
		is the compositum of $L_{m_1}$ and $L_{m_2}$. Therefore  \eqref{ah} yields
		\begin{equation}\label{ah_multiplicative}
			|\mathscr{A}_m|= [L_{m_1}L_{m_2}:\Q]=
			\frac{[L_{m_1}:\Q] [L_{m_2}:\Q]}{[L_{m_1}\cap L_{m_2}:\Q]}
		\end{equation}
		and using   \eqref{lh1}, we obtain
		\begin{equation}\label{ahmul}
			|\mathscr{A}_m|=|\mathscr{A}_{m_1}| |\mathscr{A}_{m_2}|.
		\end{equation}
		Next, note that the natural reduction map 
		$G(\Z/m\Z)\rightarrow G(\Z/m_1\Z ) \times G(\Z/m_2\Z)$ is an isomorphism, 
		where $G(R)$ denotes ${\rm GL}_2 (R) \times {\rm GL}_2 (R)$
		for any commutative ring $R$. Let $\psi$ be the 
		restriction of the above map to $\mathscr{A}_{m}\subset G(\Z/m\Z)$. We see that the image of the  map $\psi$ lies in 
		$\mathscr{A}_{m_1}\times \mathscr{A}_{m_2}$ and the map
		$$\psi:\mathscr{A}_{m} \rightarrow  \mathscr{A}_{m_1}\times \mathscr{A}_{m_2}$$
		is an injection and hence, by  \eqref{ahmul}, it is an isomorphism. 
		If we further restrict $\psi$ to $\mathscr{C}_m$, then it is easy to see that it gives a bijection of sets
		\begin{equation}\label{ch_multiplicative}
			\mathscr{C}_m \rightarrow \mathscr{C}_{m_1} \times \mathscr{C}_{m_2}.
		\end{equation}
		This completes the proof of the proposition.
	\end{proof}

	\begin{remark}\label{A_C_upper}
		In the case of higher level $N> 1$, there can be primes dividing the respective levels where the mod-$m$ representations are ramified 
		and hence the conclusion that $L_{m_1}\cap L_{m_2}=\Q$ will be false in general.  We can say, however, that the map $\psi$ is an injective homomorphism when restricted to $\mathscr{A}_{m}$ and is an injective map of sets 
		when restricted to $\mathscr{C}_{m}$ and thus we can conclude that if $(m_1, m_2)=1$,
		\be\label{submult}
		|\mathscr{A}_{m_1 m_2}|\leq |\mathscr{A}_{m_1}| | \mathscr{A}_{m_2}|  \textnormal{ and  }  |\mathscr{C}_{m_1 m_2}|\leq |\mathscr{C}_{m_1}| | \mathscr{C}_{m_2}|.
		\ee
		Also, note that for $i=1,2$, if $\psi_i: G(\Z/m_1 m_2\Z)\rightarrow G(\Z/m_i\Z)$ is the natural projection map then $\psi_i(\mathscr{A}_{m_1 m_2})=\mathscr{A}_{m_i}$ and 
		$\psi_i(\mathscr{C}_{m_1 m_2})=\mathscr{C}_{m_i}$. Therefore,
		\be\label{ulb}
		|\mathscr{A}_{m_1 m_2}|\geq |\mathscr{A}_{m_i}| \textnormal{ and  }  |\mathscr{C}_{m_1 m_2}|\geq |\mathscr{C}_{m_i}|.
		\ee
	\end{remark}

	\subsection{Sizes of  $\mathscr{A}_{m}$ and  $\mathscr{C}_{m}$ for a general integer $m$}
	By  calculations as in \S \ref {delta-asymp}, one can show that
	\begin{equation}\label{size_A_C}
		|\mathscr{A}_{\ell^n}| \ll  \ell^{7n} {\rm ~~and~~} |\mathscr{C}_{\ell^n}| \ll \ell^{5n}
	\end{equation}
	for every prime $\ell$ and every integer $n\geq 1$; and for $\ell$ large enough, we have
	\[
	|\mathscr{A}_{\ell^n}| \asymp \ell^{7n} {\rm ~~and~~} |\mathscr{C}_{\ell^n}| \asymp \ell^{5n}.
	\]
	Now let $m$  be any  positive integer. By considering its prime factorization and applying the first bounds in
	\rmkref{A_C_upper}, we obtain from \eqref{size_A_C}  the following bounds: 
	\begin{equation}\label{ch_cardinality}
		|\mathscr{A}_m| \ll m^7 {\rm ~~and~~} |\mathscr{C}_m| \ll m^5.
	\end{equation}

	\section{Asymptotic formula for $\pi_{f_1,f_2}(x,m)$ and $\pi_{f_1,f_2}^*(x,m)$}
	Let $f_1$ and $f_2$  be two non-CM eigenforms as in the previous section.
	For a positive integer $m$
	and a real number $x\ge 2$, define
	\begin{equation}\label{pi_definition}
		\pi_{f_1,f_2}(x,m):=\sum_{\substack{p\le x, (p, m N)=1 \\ m|(a_1 (p),a_2(p))}}1.
	\end{equation}
	To obtain an asymptotic formula for $\pi_{f_1,f_2}(x,m)$, our aim is to apply CDT
	for the finite Galois extension $L_m/\Q$, $L_m$ being the fixed field of the kernel of the mod-$m$ representation
	$\bar\rho_m=(\bar{\rho}_{f_1,m}, \bar{\rho}_{f_2,m})$.
	Clearly, the representation $\bar\rho_m$ is unramified at a prime $p$ such that  $(p, mN)=1$, where $N={lcm}(N_1,N_2)$, and hence  $p$ is unramified in $L_m$.  Moreover, for such a prime $p$
	$$ 
	{\rm tr}\left(\bar{\rho}_{m}({\rm Frob}_p)\right)\equiv (a_1 (p) \pmod m,~ a_2(p)\pmod m).
	$$
	Thus, we can write
	\begin{equation}\label{pi_definition1}
		\pi_{f_1,f_2}(x,m)
		= |\{p\le x: p {\rm ~unramified~in~} L_m, \bar{\rho}_m\left({\rm Frob}_p \right) \in \mathscr{C}_m \}| + O(1),
	\end{equation}
	where $\mathscr{C}_m$ is defined by  \eqref{definition_C_m} and  term $O(1)$ is to account for  the presence of  possible prime divisors of $N$ at which $\bar\rho_m$ is unramified. Note that  $\bar\rho_m$ is ramified at primes $p|m$ because a non-trivial power of the mod $p$ cyclotomic character is a component of its determinant, which is ramified at $p$. Now we state the main result of this section. 
	\begin{proposition}\label{asymptotic_pi}
		Let $f_1\in S_{k_1}(N_1)$ and $f_2\in S_{k_2}(N_2)$ be as before.  Let $N={lcm}(N_1,N_2)$ and  $m\ge1$ be an integer. 	Then we have:
		\begin{enumerate}
			\item[(a)]
			For  $\log x \gg m^{21}(\log (mN))^2$ 
			\begin{equation}\label{piprop}
				\pi_{f_1,f_2}(x,m)  = \delta(m)  {\rm Li}(x) +O\left( \delta(m) {\rm Li}(x^\beta)\right)+O\left( m^5x ~{\rm exp}\left(-c'\sqrt{\frac{\log x}{m^7}}\right) \right),
			\end{equation}
			where $c'>0$ is an effectively computable constant.
			\item[(b)]
			Under GRH, for any $x \geq 2$, we have
			\begin{equation}\label{evl}
				\pi_{f_1,f_2}(x,m)=\delta(m){\pi(x)}+O\left(m^5x^{{1}/{2}} \log(m Nx)\right).
			\end{equation}
		\end{enumerate}
		Here the $O$-constants are absolute in both {\rm (a)} and {\rm (b)}. 
	\end{proposition}
	\begin{proof}
		Note that the map $\bar{\rho}_m$ descends to an isomorphism $\tilde{\rho}_m: G \xrightarrow{\sim} \mathscr{A}_m$, where $G$ denotes
		$\textnormal{Gal}(L_m/\Q)$. We take $C$ to be the subset
		\[
		C:=\{\sigma \in G: \tilde{\rho}_m (\sigma) \in \mathscr{C}_m\},
		\]
		of $G$ which is clearly conjugacy-invariant. Now we apply  CDT (see  \eqref{cdt} and \eqref{effective_cdt}) to obtain 
		the two statements above.  Note that the number of conjugacy classes in $C$ is at most $|C|=|\mathscr{C}_m|$. 
		Here we have used the fact that $n_{L_m} \ll m^7$. This follows  from combining  \eqref{ah} and \eqref{ch_cardinality}. We have also used  a consequence of an inequality of Hensel (see \cite[Prop. 5, p. 129]{ser}) that says,
		\begin{equation}\label{Hensel}
			\log d_{L_m}\le |\mathscr{A}_m| \log(mN |\mathscr{A}_m|).
		\end{equation}
	\end{proof}

	Next we define
	\begin{equation*}
		\pi_{f_1,f_2}^*(x,m)=|\{p\le x:  a_1(p) a_2(p)\ne 0,~
		a_1(p)\equiv a_2(p)\equiv 0 \pmod m\}|.
	\end{equation*}
	It is well known (see \cite[p. 175]{ser}) that 
	\begin{equation}\label{zero_coefficient}
		|\{p \le x: a_i(p)=0  \}| =\begin{cases}
			O\left( \frac{x}{(\log x)^{3/2-\ve}}\right), & {\rm ~for ~any~} \ve >0,\\
			O(x^{3/4}), & {\rm ~under ~GRH};
		\end{cases}
	\end{equation}
	and hence using Prop. \ref{asymptotic_pi} we conclude the following.
	\begin{proposition}\label{asymptotic_pi*}
		Let the assumptions be as in Prop. \ref{asymptotic_pi}.
		Then
		\begin{enumerate}
			\item[(a)]
			for  $\log x \gg m^{21}(\log (mN))^2$ 
			\begin{equation*}
				\pi_{f_1,f_2}^*(x,m)=\delta(m) {\rm Li}(x) +O\left(\delta(m) {\rm Li}(x^\beta)\right) +O\left( \frac{x}{(\log x)^{3/2-\ve}}\right), 
			\end{equation*}
			for any small $\ve>0$.
			\item[(b)]
			Under GRH, we have 
			\begin{equation*}\label{me}
				\pi_{f_1,f_2}^*(x,m)=\delta(m)\pi(x)+O\left(m^5x^{{1}/{2}}\log (mNx)\right)+O(x^{{3}/{4}}).
			\end{equation*}
		\end{enumerate}
		Here the $O$-constants are absolute in both {\rm (a)} and {\rm (b)}. 
	\end{proposition}
	
	\begin{remark}\label{difficulty}
		Note that the error terms that appear in the conditional versions of the above propositions are quite large in terms of $m$. This makes \
		handling the sum of $\omega((a_1 (p), a_2 (p)))$ over primes $p$ difficult since the technique we use in proving Thm. \ref{omega} requires summing 
		$\pi_{f_1,f_2}^*(x,\ell)$ over primes $\ell$. This is the reason we are unable to obtain an asymptotic formula in Thm. \ref{omega}. Herein  lies also
		the difficulty in proving Conjecture \eqref{conj}. 
	\end{remark}

	\section{Preparation for the proof of \thmref{main}}\label{S_preparation}
	
	The bound in  \eqref{Hensel}, together with
	\eqref{beta_stark}, implies that the Landau-Siegel zero $\beta$, if it exists, satisfies
	\begin{equation*}
		\beta \le 1-\frac{c_0}{Nm^8},
	\end{equation*}
	where $c_0$ is as in  \eqref{beta_stark}. Therefore, in such a case we can choose a  constant $c>0$ such that
	\begin{equation}\label{beta}
		\beta \le 1-\frac{1}{m^c}
	\end{equation}
	uniformly over all $m\geq 2$ and we can  assume without loss of generality that $c\geq 21$. For example, we can take
	\[
	c=\max\left\{21, 9+ \frac{|\log (N/c_0)|}{\log 2}\right\}.
	\]	
	\noindent
	For this and the next section only, we set
	\begin{equation}\label{y}
		y =y(x)= L_2(x)^\eta, {\rm~~where~~} 0<\eta <\min\left\{\frac{1}{21},\frac{1}{c-2}\right\}.
	\end{equation}
	Here $c$ is as above and if  $\beta$ does not appear in Prop. \ref{asymptotic_pi}, then we set $c=21$.
	
	Below we will frequently use standard estimates such as
	\[
	\sum_{p\leq x} \frac 1  p \ll L_2(x), \textnormal{  or  } \sum_{p \leq x} \frac {\log p}{p} \ll \log x.
	\]
	\begin{lemma}\label{sum_1_over_p}
		With the assumptions and notation as above, we have
		\begin{enumerate}
			\item[(a)]
			for any $ \ell \le y$
			\begin{equation}\label{1bypno}
				\sum_{\substack{p\le x \\ \ell  |(a_1(p),a_2(p))}}\frac{1}{p}
				=\delta(\ell )L_2(x)+O\left(\frac{\ell^{c-2}}{L_2(x)} \right)
				+O(L_3(x)).
			\end{equation}
			\item[(b)]
			Under GRH
			\begin{equation*}\label{1bypgrh}
				\sum_{\substack{p\le x \\ \ell |(a_1(p),a_2(p))}}\frac{1}{p}= \delta(\ell)L_2(x)
				+O \left(\ell^5 \log \ell \right).
			\end{equation*}
		\end{enumerate}
	\end{lemma}
	\begin{proof}
		By partial summation, we write
		\begin{align*}
			\sum_{\substack{p\le x \\ \ell |(a_1(p),a_2(p))}}\frac{1}{p}&= \frac{1}{x}{\pi_{f_1,f_2}(x,\ell )}
			+
			\int_2^x \frac{1}{t^2}  \pi_{f_1,f_2}(t,\ell ) dt +O(1),
		\end{align*}
		where the error term is present because of the primes dividing $N\ell$. Since
		$$
		\frac{1}{x}{\pi_{f_1,f_2}(x,\ell )}\le  \frac{1}{x}\pi(x)=O(1),
		$$ 
		we have,
		\begin{align}\label{sum1}
			\sum_{\substack{p\le x \\ \ell |(a_1(p),a_2(p))}}\frac{1}{p}&= \int_2^x \frac{1}{t^2}  \pi_{f_1,f_2}(t,\ell ) dt +O(1).
		\end{align}
		Now  subdividing the interval $[2, x]$, we write
		\begin{equation*}
			\sum_{\substack{p\le x \\ \ell |(a_1(p),a_2(p))}}\frac{1}{p} =    \int_2^T \frac{1}{t^2} \pi_{f_1,f_2}(t,\ell )dt+
			\int_{T}^x \frac{1}{t^2} \pi_{f_1,f_2}(t,\ell )dt+ O(1),
		\end{equation*}
		where $T$ is to be chosen later. The first integral on the right hand side is
		$$
		\ll   \int_2^T \frac{1}{t^2} \pi(t)dt \ll L_2(T).
		$$
		To estimate the second integral, we assume that $T$ is  large enough so that  \eqref{piprop} can be applied for  $h=\ell$; i.e., we need
		\[
		{\ell}^{21}(\log (\ell N))^2 \ll \log T.
		\]
		Since  $y$ satisfies \eqref{y}, we have,
		$${\ell}^{21}(\log (\ell N))^2 \leq y^{21}(\log (yN))^2 \ll L_2(x)^{21 \eta }L_3(x)^2\ll L_2(x)^{1-\ve}$$ for some suitable $\ve>0$; and hence,
		we can take $T= \log x$. 
		Therefore, the second integral is
		$$\int_{\log x}^x  \frac{1}{t^2} \left(\delta(\ell)  {\rm Li}(t) +O\left( \delta(\ell) {\rm Li}(t^\beta)\right)+O\left({\ell}^5 t~{\rm exp}\left(-c'\sqrt{\frac{\log t}{{\ell}^7}}\right) \right)\right)dt.
		$$
		Now we see that
		\begin{align*}
			\int_{\log x}^x   \delta(\ell )  \frac{{\rm Li}(t)}{t^2} dt
			&= \int_{\log x}^x  \delta(\ell)\left(\frac{1}{t\log t}+ O\left( \frac{1}{t(\log t)^2}\right)\right) dt\\
			&= \delta(\ell)L_2(x) +O(L_3(x) ).
		\end{align*}
		Next we consider the integral involving $\beta$. From  \eqref{beta} we know that $\beta \le 1-\frac{1}{{\ell}^c}$. Therefore,
	\begin{align*}
		\int_{\log x}^x \frac{{\rm Li}(t^\beta)}{t^2}dt &\ll \int_{\log x}^x \frac{1}{t^{2-\beta}\log t }dt
		\ll \int_{\log x}^x \frac{1}{t\log t }\exp \left(-\frac{\log t}{{\ell}^c}\right)dt\\
		& \ll  {\ell}^c (L_2(x))^{-1}\exp \left(-\frac{L_2(x)}{{\ell}^c}\right)\ll {\ell}^c (L_2(x))^{-1}.
	\end{align*}
	Finally,
	\begin{align*}
		\int_{\log x}^x \ell^5 \frac{1}{t} ~{\rm exp}\left(-c'\sqrt{\frac{\log t}{\ell^7}}\right)dt
		&= \ell^5 \int_{L_2(x)}^{\log x} {\rm exp}\left(-c'\sqrt{\frac{u}{\ell^7}} \right)du \\
		& \ll  \ell^{5+7/2}\sqrt{ L_2(x)}{\rm exp}\left(-c'\sqrt{\frac{L_2(x)}{\ell^7}}\right)\\
		&\ll \ell^{19} (L_2(x))^{-1}
	\end{align*}
		and this completes the proof of part (a).\\
		
		The proof of part (b) is very similar after applying  \eqref{evl} 
		in  \eqref{sum1}. 
	\end{proof}

	Given a positive integer $n$ and a prime $\ell$, we define
	\begin{equation}\label{vd}
		v(\ell,n)=|\{p^{\alpha}:p^{\alpha}\Vert n ~{\rm and}~\ell|\left(a_1 (p^{\alpha}),a_2(p^{\alpha})\right)\}|.
	\end{equation}
	In the next two lemmas, we obtain asymptotic formulae for the partial sums of $v(\ell,n)$ and $v^2(\ell,n)$ which will play an important role in proving  Thm. \ref{main}.
	\begin{lemma}\label{sum_v} 
		With the same assumption and notation as above, we have,
		\begin{enumerate}
			\item[(a)]\label{a}
			for any $ \ell \le y$, 
			\begin{equation*}\label{wos}
				\sum_{n\le x}v(\ell,n)=\delta( \ell )xL_2(x)
				+O(xL_3(x));
			\end{equation*}
			\item[(b)]
			and under GRH, we have, for any prime $\ell$,
			\begin{equation*}\label{wosgrh}
				\sum_{n\leq x}v(\ell ,n)=\delta(\ell){x}L_2(x)+O\left(\ell^5x\log \ell \right).
			\end{equation*}	
		\end{enumerate}
	\end{lemma}
	\begin{proof}
		We write 
		\begin{align*}\label{v1}
			\sum_{n\le x}v(\ell ,n)&=\sum_{n\le x}\sum_{\substack{p^{\alpha}\Vert n\\ \ell |(a_1  (p^{\alpha}), a_2 (p^{\alpha}))}} 1
			=\sum_{\substack{p^{\alpha}\le x\\ \ell |(a_1  (p^{\alpha}), a_2 (p^{\alpha}))}}\sum_{\substack{n\le x\\p^{\alpha}\Vert n }}1.
		\end{align*}
		We split the  sum into two parts, the one with $\alpha=1$ and the other with $\alpha 
		\geq 2$.
		Since the contribution from all the terms with $\alpha\ge 2$ is $O(x)$, so we can write
		\begin{equation}\label{v1}
			\sum_{n\le x}v(\ell,n) =\sum_{\substack{p\le x\\ \ell |(a_1  (p), a_2 (p))}}\sum_{\substack{n\le x\\p\Vert n }}1 +O(x).
		\end{equation}
		Simplifying further, we use 
		the easily proved asymptotic formula
		$$\sum_{\substack{n\le x\\p\Vert n }}1=\frac{x}{p}+O\left(\frac{x}{p^2}\right) +O(1)$$  
		to obtain
		\begin{align*}
			\sum_{\substack{p\le x\\ \ell  |(a_1  (p), a_2 (p))}}\sum_{\substack{n\le x\\p\Vert n }}1
			=\sum_{\substack{p\le x\\ \ell |(a_1  (p), a_2 (p))}}\left\{\frac{x}{p}+O\bigg(\frac{x}{p^2}\bigg)+O(1)\right\}
			=x \Big( \sum_{\substack{p\le x\\ \ell |(a_1  (p), a_2 (p))}}\frac{1}{p}\Big)+O(x).
		\end{align*}
		Now use \lemref{sum_1_over_p} for the sum appeared in the right hand side of this expression
		to obtain
		\begin{equation}\label{v4}
			{\displaystyle \sum_{\substack{p\le x\\ \ell |(a_1  (p), a_2 (p))}}\sum_{\substack{n\le x\\p\Vert n }}1=
				\begin{cases}
					{\displaystyle	\delta(\ell )xL_2(x)+O\left(\frac{\ell^{2c-2}x}{L_2(x)} \right) 
						+O(xL_3(x))}, &\ell \le y,\\
					{\displaystyle 	\delta(\ell){x} L_2(x)+O\left(\ell^5 x\log \ell \right)},& {\rm under ~GRH}.
			\end{cases}}
		\end{equation}
		Finally, substituting  \eqref{v4} in  \eqref{v1}  and recalling the choice of $y$ in \eqref{y}, we finish the proof.
	\end{proof}

	\begin{lemma}\label{sum_v_2}
		We have,
		\begin{enumerate}
			\item[(a)]
			for any $ \ell \le y$,
			\begin{equation*}\label{ws}
				\sum_{n\le x}v^2( \ell ,n)= \delta( \ell)^2 x(L_2(x))^2 + O\left(\delta(\ell) xL_2(x) L_3(x) \right);
			\end{equation*}
			\item[(b)]
			and under GRH, we have, for any prime $\ell$, 
			\begin{equation*}
				\sum_{n\leq x}v^2(\ell,n)=\delta(\ell)^2{x}(L_2(x))^2 
				+O\left(\ell^{10}x (\log \ell )^2 \right) 
				+O\left(\ell^3 x L_2(x)  \log \ell \right).
			\end{equation*}
		\end{enumerate}
	\end{lemma}
	
	\begin{proof}
		We have,
		\begin{align*}
			\sum_{n\leq x}v^2( \ell ,n)
			=	\sum_{n\le x}\sum_{\substack{p^{\alpha}\Vert n,  q^{\beta}\Vert n\\  \ell |(a_1 (p^{\alpha}), a_2(p^{\alpha}))
					\\  \ell |(a_1 (q^{\beta}), a_2(q^{\beta}))}} 1.
		\end{align*}
		We split the above sum into three parts: the first one with $\alpha=\beta=1$, the second one with exactly one of $\alpha$ and $\beta$ $=1$; and the third one with $\min\{\alpha, \beta\}>1$. In view of \lemref{sum_1_over_p}, the second sum will contribute $O(xL_2(x))$ whereas  the contribution from the last sum  is $O(x)$. Therefore, introducing the notation
		\[
		D(\ell):=\{(p,q): \ell |(a_1 (p), a_2(p)),  \ell |(a_1 (q), a_2(q)\},
		\]
		we may write
		\begin{align}\label{vs11}
			\sum_{n\leq x}v^2( \ell ,n)& =\sum_{n\le x}\sum_{\substack{(p, q)\in D(\ell) \\p\Vert n, q\Vert n}} 1+ O(xL_2(x)) \notag \\
			&= \sum_{n\le x}\sum_{\substack{p\le x, p\Vert n\\  \ell |(a_1 (p), a_2(p))}} 1 +
			\sum_{n\le x}\sum_{\substack{ (p, q)\in D(\ell)\\p\neq q, p\Vert n, q\Vert n}} 1+O(xL_2(x)).
		\end{align}
		The first sum on the right of the above equation is already examined  in  \eqref{v4}. So
		we now simplify the latter sum. 
		\begin{align}\label{p_not_q_1}
			\sum_{n\le x}\sum_{\substack{ (p, q)\in D(\ell)\\p\neq q, p\Vert n, q\Vert n}} 1=&
			\sum_{\substack{(p, q)\in D(\ell)\\pq \leq x}} 
			\sum_{\substack{n\le x\\p\neq q, p\Vert n, q\Vert n}}1 \notag\\
			=&
			\sum_{\substack{(p, q)\in D(\ell)\\pq \leq x, p\neq q}} 
			\left\{\frac{x}{pq}+O\bigg(x\bigg(\frac{1}{p^2q}+\frac{1}{pq^2}\bigg)\bigg)+O(1)\right\}\notag\\
			=&
			\sum_{\substack{(p, q)\in D(\ell)\\pq \leq x, p\neq q}} 
			\frac{x}{pq}+O\bigg(x\sum_{p\leq x}
			\frac{1}{p}\bigg) +O\left(xL_2(x)\right).
		\end{align}
		We can express the first sum as (this trick is called the Dirichlet hyperbola method)
		\begin{align}\label{Dirichlet}
			\sum_{\substack{(p, q)\in D(\ell)\\pq \leq x, p\neq q}} 
			\frac{1}{pq}=&2\sum_{\substack{p\leq \sqrt{x}\\  \ell |(a_1 (p), a_2(p))}}\frac{1}{p}\sum_{\substack{q\le \frac{x}{p}, p\neq q\\  \ell |(a_1 (q), a_2(q))}}\frac{1}{q}-\bigg(\sum_{\substack{p\leq \sqrt{x}\\  \ell |(a_1 (p), a_2(p))}}\frac{1}{p}\bigg)^2
		\end{align}
		and using part (a) of \lemref{sum_1_over_p} for each individual sum, we obtain
		\begin{align}\label{p_not_q_2}
			\sum_{\substack{(p, q)\in D(\ell)\\pq \leq x, p\neq q}} 
			\frac{1}{pq}=\delta(\ell)^2L_2(x)^2+O(\delta(\ell)L_2(x)L_3(x))
			+O\left(\ell^{2c-2}\right).
		\end{align}
		Noting that $\ell \leq y$ and combining \eqref{y},  \eqref{vs11}, \eqref{p_not_q_1} and \eqref{p_not_q_2}  completes the proof of part (a).\\
		
		To prove part (b),  we  use \lemref{sum_1_over_p} (b)  in \eqref{Dirichlet} and proceed as before.
	\end{proof}

	\section{Proof of \thmref{main}}\label{S_main}
	Recall that the parameter $y$ is defined by  \eqref{y}.\\
	To prove part (a), we first write
	\begin{equation}\label{eq110}
		\sum_{\substack{n\leq x\\ (n, (a_1(n), a_2(n)))=1}}1=\sum_{\substack{n\leq x\\ (n, (a_1(n), a_2(n)))=1\\ p|n \implies p>y}}1+\sum_{\substack{n\leq x\\ (n, (a_1(n), a_2(n)))=1\\  \ell |n {\rm ~for~some} ~\ell \leq y }}1.
	\end{equation}
	We denote the first and the second sum on the right hand side of \eqref{eq110} by $S_1$ and $S_{2}$, respectively.	Now to estimate $S_1$, we need the following standard and  easily proved lemma (the sieve of Eratosthenes).
	\begin{lemma}
		For any real numbers $x\geq 3$ and $y\geq 2$, we have
		\be\label{erato}
		\sum_{\substack{1\leq n \leq x\\ p|n \implies p>y}} 1 =x\prod_{p\leq y} \Big(1-\frac{1}{p}\Big)  +O(2^y).
		\ee
	\end{lemma} 
	Using the above lemma, we can now write
	\begin{equation}\label{eq112}
		S_1
		\ll \sum_{\substack{n\leq x\\ p|n \implies p>y}}1
		\ll x \prod_{p\leq y}\Big(1-\frac{1}{p}\Big) \ll \frac{x}{L_3(x)}.
	\end{equation}
	For estimating $S_{2}$, we first note that the Fourier coefficients are multiplicative since $f_1$ and $f_2$  are both eigenforms. It follows that if $(n, (a_1(n),a_2(n)))=1$, then  $v(\ell,n)=0$ for all primes $\ell|n$, where 
	$v(\ell,n)$ is defined by  \eqref{vd}. This assertion will be false if we do not have multiplicativity over the entire set of positive integers and this is why we
	need to restrict to forms that are eigenfunctions of all the Hecke operators. 
	Thus
	\begin{equation}\label{secondsum}
		S_{2}=\sum_{\substack{n\leq x\\ (n, (a_1(n), a_2(n)))=1\\  \ell |n {\rm ~for~some} ~\ell \leq y }}1\le \sum_{\ell \le y}\sum_{\substack{n\le x, \ell |n\\ v(\ell ,n)=0}}1 
		\le \sum_{\ell \le y}\sum_{\substack{n\le x \\  v(\ell ,n)=0}}1.
	\end{equation}
	The idea for estimating the inner sum $\displaystyle{\sum_{\substack{n\le x \\  v(\ell ,n)=0}}1}$ for a given prime $\ell\le y$ is encapsulated in the following 
	simple yet crucial lemma. The idea of this lemma is not new. See, e.g.,  \cite{km}. 
	\begin{lemma}\label{erdos}
		Suppose $(a_n)$ is a sequence of real numbers such that 
		\[
		\sum_{n \leq x}a_n =c(x) x +E_1(x) {~ ~ and ~ ~}
		\sum_{n \leq x}a_n^2 =c(x)^2 x+E_2(x),
		\]
		for large enough $x$,
		where we assume that $c(x)$ is a function that never vanishes. 
		Then  we have, 
		\[
		\displaystyle{\sum_{\substack{n\le x \\  a_n=0}}1}\leq c(x)^{-2}\left(E_2(x)-2c(x)E_1(x)\right). 
		\]
	\end{lemma}
	\begin{proof}
		This is clear once we observe that 
		\[
		\sum_{n \leq x} (a_n -c(x))^2 \geq \sum_{\substack{n\le x \\  a_n=0}}c(x)^2,
		\]
		by non-negativity. 
	\end{proof}
	As mentioned in \rmkref{delta_positive}, we recall that for any prime $\ell$, the image of complex conjugation lies in $\mathscr{C}_\ell$ because the trace of this image is always zero. This shows that $\mathscr C_\ell \neq \phi$ and hence $\delta(\ell)$ is never zero. Now applying the above lemma with $a_n=v(\ell, n)$ in conjunction  with 
	\lemref{sum_v} and \lemref{sum_v_2} (with $c(x)=\delta(\ell)L_2(x)$), we obtain 
	\begin{equation}\label{vln=0}
		\sum_{\substack{n\leq x \\ v(\ell,n)=0}} 1 \ll 
		\frac{x L_3(x)}{\delta(\ell)L_2(x)}.
	\end{equation}
	Substituting the above bound in  \eqref{secondsum}, recalling that  $\delta(\ell) \sim {\ell}^{-2}$ and the size of the parameter $y$ given in  \eqref{y}, we obtain
	\be\label{eq11}
	S_{2}  \ll \frac{x}{{L_2(x)}^{6/7}}.
	\ee
	Finally, substituting estimates \eqref{eq112} and \eqref{eq11} in  \eqref{eq110} completes the proof.\\
	
	To prove part (b), we first note that  if $(d, (a_1(n),a_2(n)))=1$, then  $v(\ell,n)=0$ for all primes $\ell|d$. Thus
	\begin{align*}
		\sum_{\substack{n\leq x\\ (d, (a_1(n), a_2(n)))=1}}1\le 
		\sum_{\substack{n\le x\\ v(\ell ,n)=0 {\rm~for~all~}\ell|d}}1 
		= \sum_{\ell |d}\sum_{\substack{n\le x \\  v(\ell ,n)=0}}1. 	
	\end{align*}
	Now  \eqref{vln=0} yields the result.
	
		\section{Proof of Theorem \ref{sieve}}\label{proof-sieve}
	
	We shall denote $(a_1(p), a_2(p))$ by $a_p$. Our goal is to estimate the number of primes $p$ up to $x$ for which $a_p=1$. Motivated by the theory of sieves, we  first make
	a simple yet crucial observation:
	\[
C(x;f_1,f_2)=|\{p\leq x: a_p=1\}| \leq | \{p\leq x: (a_p, P(y))=1\}|,
	\]
	where $P(y)=\prod_{\ell <y}\ell$, $\ell$ running over primes, for some parameter $y$ to be chosen later, subject to the conditions $y<x$ and that $y$ goes to infinity along with $x$.  Our goal is to estimate the sum on the right accurately. This is usually done using sieves. However, since the density function $\delta$ is not known to multiplicative on the entire set of integers in the general case, we cannot apply standard results from Sieve Theory directly. Instead, we start from the scratch by expressing the coprimality condition by the
	M\"{o}bius function.
Thus we write,
	\begin{align*}
	\sum_{\substack{p\leq x\\(a_p, P(y))=1}} 1&= \sum_{p \leq x} \sum_{d|(a_p, P(y))}\mu(d)\\
	&= \sum_{d|P(y)}\mu(d)\sum_{\substack{p\leq x\\a_p\equiv 0 \modd}} 1.
	\end{align*}
	Now, under the assumption of GRH, we have, by \eqref{evl},
	\[
	\sum_{\substack{p\leq x\\a_p\equiv 0 \modd}} 1= \delta(d)\pi(x)+O\left(d^5x^{1/2}\log (dNx)\right).
	\]
Therefore,
\be\label{one}
\sum_{\substack{p\leq x\\(a_p, P(y))=1}} 1=\pi(x)\sum_{d|P(y)}\mu(d)\delta(d)+
O\left(x^{1/2} \sum_{d|P(y)}\mu^2(d)d^5 \log (dNx)\right).
\ee
We first treat the sum in the error term. First of all, for $d|P(y)$,
\[
\log d \leq \log P(y)=\sum_{\ell <y}\log \ell \ll y,
\]
by PNT. Therefore, 
\[
\sum_{d|P(y)}\mu^2(d)d^5 \log (dNx)\ll (y +\log x)\sum_{d|P(y)}\mu^2(d)d^5.
\]
Now,
\begin{align*}
\sum_{d|P(y)}\mu^2(d)d^5 &= \prod_{\ell <y}(1+\ell)^5\\
&\ll \prod_{\ell<y}{\ell}^5\\
&\ll\exp((5+\ve)y),
\end{align*}
for any fixed $\ve>0$,
again by PNT. Therefore,
\be\label{two}
x^{1/2} \sum_{d|P(y)}\mu^2(d)d^5 \log (dNx)=O\left(x^{1/2}\exp((5+\ve)y)(y+\log x)\right).
\ee

We now treat the sum in the main term. Note that 
\be\label{three}
\sum_{d|P(y)}\mu(d)\delta(d)= \alpha' -\sum_{d\nmid P(y)}\mu(d)\delta(d).
\ee
Now to handle this new sum, we need to overcome the problem of the lack of multiplicativity coming from the small primes.  We first recall the definition of $M=M(f_1, f_2)$ in the beginning of \S \ref{delta-asymp} and the bound $\delta(\ell)\leq 3/{\ell}^2$
for $\ell >M$ (see \eqref{deltabd}). We observe that every $d$ in the above sum can be factored uniquely as $d=d_1 d_2$, where
\[
d_1=\prod_{\substack{p|d\\ p\leq M}} p
\]
and $d_2=d/d_1$. We also make another observation that for any two positive integers $a$ and $b$, $\delta(ab)\leq \delta(a)$. This does not follow 
directly from the definition of $\delta$ but the observation is clear once we interpret $\delta$ as a density using Prop. \ref{asymptotic_pi}; namely,
\[
\delta(a)=\lim_{x \rightarrow \infty} \frac{\pi_{f_1, f_2} (x, a)}{\pi (x)},
\]
since, trivially, $\pi_{f_1, f_2} (x, ab)\leq \pi_{f_1, f_2} (x, a)$.  Using the above observations and recalling the standard notations $P^+ (n)$ and $P^-(n)$ for the largest and the smallest prime factor of a positive integer $n$, respectively, we write
\begin{align*}
|\sum_{d\nmid P(y)}\mu(d)\delta(d)|&\leq \sum_{d\nmid P(y)}\mu^2(d)\delta(d)\\
&=\sum_{P^+ (d)>y}\mu^2(d)\delta(d)\\
&=\sum_{P^+(d_1)\leq M}\sum_{\substack{P^-(d_2)> M\\P^{+}(d_2)>y}}\mu^2(d_1 d_2)\delta(d_1d_2)\\
&\leq \sum_{P^+(d_1)\leq M} \mu^2(d_1)\sum_{\substack{P^-(d_2)> M\\P^{+}(d_2)>y}}\mu^2( d_2)\frac{3^{\omega(d_2)}}{{d_2}^2}\\
&\leq 2^M \sum_{c>y}\frac{3^{\omega(c)}}{{c}^2}\\
&\ll y^{-1}(\log y)^2,\\
\end{align*}
by a well-known classical estimate and partial summation.
By this bound and \eqref{one}, \eqref{two}, and \eqref{three}, we finally obtain 
\[
\sum_{\substack{p\leq x\\(a_p, P(y))=1}} 1=\alpha'\pi(x)+O(y^{-1}(\log y)^2\pi(x))+O\left(x^{1/2}\exp((5+\ve)y)(y+\log x)\right).
\]
Now, if we choose $y=\frac{1}{12}\log x$ and $0<\ve<1/100$, we see that both the error terms are $o(\pi(x))$. 

	\section{Proofs of \thmref{omega} and \thmref{omega_u}}\label{S_omega}
	For this section, we set $z=x^{1/12-\eta}$ for some fixed real number $\eta \in (0, 1/100)$. We first prove  two lemmas.
	\begin{lemma}\label{pi_del}
		Under GRH, we have
		$$\sum_{\ell\leq z}|\pi^{*}_{f_1,f_2}(x,\ell)-\delta(\ell)\pi(x)|=o(\pi(x)).$$
	\end{lemma}
	\begin{proof}
		From part (b) of  Prop. \ref{asymptotic_pi*} we obtain
		$$\sum_{\ell\leq z}|\pi^{*}_{f_1,f_2}(x,\ell )-\delta(\ell )\pi(x)|=O\Big(x^{\frac{1}{2}}\sum_{\ell \leq z}\ell ^5\log (\ell Nx)\Big)+O\Big(x^{\frac{3}{4}}\sum_{\ell \leq z}1\Big).
		$$
		Because of our choice of $z$,  both the error terms  on the right hand side are $o(\pi(x))$ and this completes the proof.
	\end{proof}
	
	\begin{lemma}\label{sum_pi_order}
		Under GRH, we have
		\begin{equation*}
			\sum_{\ell}\pi^{*}_{f_1,f_2}(x,\ell )\ll \frac{x}{\log x}.
		\end{equation*}
	\end{lemma}
	
	\begin{proof}
		First, we write
		$$\sum_{\ell}\pi^{*}_{f_1,f_2}(x,\ell )=\sum_{\ell \leq z}\pi^{*}_{f_1,f_2}(x,\ell )+\sum_{ \ell>z}\pi^{*}_{f_1,f_2}(x,\ell ).$$
		Since
		$$\sum_{\ell \le z}\pi^{*}_{f_1,f_2}(x,\ell )\le \sum_{\ell \le z}|\pi^{*}_{f_1,f_2}(x,\ell )-\delta(\ell )\pi(x)|
		+\sum_{\ell \le z}\delta(\ell )\pi(x),
		$$
		applying \lemref{pi_del},  Prop. \ref{asymptotic_delta}, and PNT, we obtain the bound
		\begin{equation*}
			\sum_{\ell \le z}\pi^{*}_{f_1,f_2}(x,\ell )\ll \frac{x}{\log x}.
		\end{equation*}	
		Thus in order to complete the proof it suffices to show that
		$\displaystyle{\sum_{\ell>z}\pi^{*}_{f_1,f_2}(x,\ell )\ll \frac{x}{\log x}.}$
		Now
		$$\sum_{ \ell>z}\pi^{*}_{f_1,f_2}(x,\ell )
		= \sideset{}{'}\sum_{p \leq x} \sum_{\substack{\ell>z  \\ \ell |(a_1(p), a_2(p))}}1,$$
		and using the fact that $(a_1(p), a_2(p))\ll x^{(k-1)/{2}}$ for $p \leq x$ we have
		\begin{equation}\label{nd}
			\sum_{\substack{\ell>z  \\ \ell |(a_1(p), a_2(p))}}1 \ll \frac{\log x}{\log z}=O(1),
		\end{equation}
		which yields the desired result.
	\end{proof}
	
	\subsection{Proof of \thmref{omega}}
	We observe that under GRH, the following bound holds:
	\begin{equation}\label{omegafirst}
		\sideset{}{'}\sum\limits_{p\le x} \omega((a_1(p),a_2(p)))\ll \frac{x}{\log x}.
	\end{equation}
	Indeed,
	\begin{align*}
		\sideset{}{'}\sum\limits_{p\le x} \omega((a_1(p),a_2(p))) &
		=\sideset{}{'}\sum\limits_{p\le x} \sum_{\substack{ \ell|(a_1 (p),a_2(p))}}1 
		= \sum_{\ell } \sideset{}{'}\sum\limits_{{\substack{p \le x \\ \ell|(a_1 (p),a_2(p))}}}1
		= \sum_{\ell } \pi_{f_1,f_2}^*(x,\ell).
	\end{align*}
	Now the estimate \eqref{omegafirst} is clear after  invoking \lemref{sum_pi_order}. \\
	By the elementary inequality $(a+b)^2\leq 2(a^2 +b^2)$ for any two real numbers $a$ and $b$, we can write
	\begin{align*}
		\sideset{}{'}\sum\limits_{p\le x} \omega^2((a_1(p),a_2(p)))\ll
		\sideset{}{'}\sum\limits_{p\le x} \left( \omega((a_1(p),a_2(p))) - \omega_{\sqrt z}((a_1(p),a_2(p))) \right)^2  + \sideset{}{'}\sum\limits_{p\le x} \omega_{\sqrt z}^2((a_1(p),a_2(p))).
	\end{align*}
	Now,
	$\omega((a_1(p),a_2(p))) - \omega_{\sqrt z}((a_1(p),a_2(p)))$ is the number of distinct prime divisors of $(a_1(p),a_2(p))$ lying between
	$\sqrt{z}$ and $2x^{(k-1)/2}$
	and hence from \eqref{nd} 
	$$\sideset{}{'}\sum\limits_{p\le x} \left( \omega((a_1(p),a_2(p))) - \omega_{\sqrt z}((a_1(p),a_2(p))) \right)^2\ll \frac{x}{\log x}.$$
	Thus to complete the proof it remains to show that 
	$$
	\sideset{}{'}\sum\limits_{p\le x}  \omega^2_{\sqrt z}((a_1(p),a_2(p))) \ll \frac{x}{\log x}
	$$
	and this follows from  part (b) of Thm. \ref{omega_u} which is proved in the next section.
	\subsection{Proof of \thmref{omega_u}}
	To prove part (a), we write
	\begin{align}\label{id_effective}
		\sideset{}{'}\sum\limits_{p\le x} \omega_u((a_1(p),a_2(p)))
		&= \sum_{\ell \le u} \pi_{f_1,f_2}^*(x,\ell) \notag\\
		&= \pi(x) \left(\sum_{\ell} \delta(\ell)- \sum_{\ell > u} \delta(\ell) \right)+ \sum_{\ell \le u} \left( \pi_{f_1,f_2}^*(x,\ell)-\delta(\ell)\pi(x)\right).
	\end{align}
	By Prop. \ref{asymptotic_delta}, we know that the series $\sum_{\ell} \delta(\ell)$ is convergent and we denote the sum by $c_1$. Obviously, $c_1>0$. Also, by Prop. \ref{asymptotic_delta} and partial summation, we have the bound
	$$\sum_{\ell > u} \delta(\ell) \ll \frac{1}{u}.$$
	Thus applying Prop. \ref{asymptotic_pi*} in  \eqref{id_effective} and using  PNT, we obtain
	\begin{align*}
		\sideset{}{'}\sum\limits_{p\le x} \omega_u((a_1(p),a_2(p)))
		&= c_1\pi(x) +O\left(\frac{x}{u \log x}\right) + O\left(x^{{1}/{2}}(\log x) \frac{u^6}{\log u}\right)+O\left(x^{{3}/{4}}\frac{u}{\log u}\right).
	\end{align*}
	This completes the proof because of our choice of $u$.
	The proof of part (b) is omitted as one just needs to follow the same idea that has been used for proving part (a). 

	\begin{acknowledgements} 
		The authors thank E. Ghate,  V. M. Patankar,  C. S. Rajan and J. Sengupta for helpful discussions. The authors thank the anonymous referee for a careful reading of the manuscript and they are grateful for
		several suggestions from the referee that  led to a substantial improvement in the quality of this article. This project was initiated when
		the first named author visited the Tata Institute of Fundamental Research, Mumbai in January, 2019 where  the second and the third authors were postdoctoral fellows at that time. The authors  thank 
		the institute for providing excellent working condition. 	 The open-source mathematics software SAGE (www.sagemath.org) has been used 
		for numerical computations in this work.\\
		The research of the second author was supported by the grant no. 692854 provided by the European Research Council (ERC) while the third author was supported by Israeli Science Foundation grant  1400/19. 
	\end{acknowledgements}

\end{document}